\documentclass[11pt,a4paper]{article}
\usepackage{amsmath,amsthm,amssymb,amsfonts,amstext,graphics,dsfont,euscript,enumerate,multirow}
\usepackage[curve]{xypic} 
\usepackage[latin1]{inputenc}
\usepackage[colorlinks=true,pdfstartview=FitV,linkcolor=blue,citecolor=blue,urlcolor=blue]{hyperref}
\usepackage[usenames,dvipsnames,svgnames,table]{xcolor}
\usepackage{makeidx}
\usepackage{url}

\newcommand{\mb}{\mathds}

\newcommand{\mc}{\mathcal}
\newcommand{\ra}{\rightarrow}
\newcommand{\lra}{\longrightarrow}

\newcommand{\A}{\mathbf{A}}
\newcommand{\B}{\mathbf{B}}

\newcommand{\D}{\mathbf{D}}

\newcommand{\E}{\mathbf{E}}

\newcommand{\Dcris}{\D_{\mbox{\scriptsize cris}}}

\newcommand{\Ddr}{\D_{\mbox{\scriptsize dR}}}

\newcommand{\Bcris}{\B_{\mbox{\scriptsize cris}}}
\newcommand{\Acris}{\A_{\mbox{\scriptsize cris}}}

\newcommand{\Bdr}{\B_{\mbox{\scriptsize dR}}}

\newcommand{\Bplusdr}{\B_{\mbox{\scriptsize dR}}^+}

\newcommand{\Bpluscris}{\B_{\mbox{\scriptsize cris}}^+}

\newcommand{\Etplus}{\widetilde{\E}^+}

\newcommand{\Atplus}{\widetilde{\A}^+}

\newcommand{\Hom}{\trm{Hom}}

\newcommand{\trm}{\textrm}
\newcommand{\Fil}{\textrm{Fil}}

\newcommand{\bmat}{\begin{pmatrix}}
\newcommand{\emat}{\end{pmatrix}}
\newcommand{\Cp}{\mb{C}_p}
\newcommand{\Ocp}{\mc{O}_{\mb{C}_p}}
\newcommand{\Qp}{\mb{Q}_p}
\newcommand{\Zp}{\mb{Z}_p}

\renewcommand{\hat}{\widehat}
\renewcommand{\tilde}{\widetilde}

\newcommand{\MF}{\underline{\trm{MF}}}

\newcommand{\OK}{\mc{O}_K}
\newcommand{\Aexp}{\A_{\trm{exp}}}
\newcommand{\Dexp}{\D_{\trm{exp}}}

\makeatletter
\makeindex

\newtheorem{thm}{Theorem}[section]
\newtheorem*{thm*}{Theorem}
\newtheorem{cor}[thm]{Corollary}
\newtheorem{prop}[thm]{Proposition}
\newtheorem{lem}[thm]{Lemma}

\newtheorem{defi}[thm]{Definition}
\newtheorem{ex}[thm]{Example}

\title{Integral structures on the finite part $H^1_f(K, V)$ of a crystalline representation}

\author{Andreas Riedel}

\date{3. August, 2016}

\begin{document}

\maketitle

\begin{abstract}
  We study integral structures of crystalline representations over an
  unramified extension $K / \Qp$ with the help of an auxillary ring
  $\Aexp$. This ring has the nice property that it contains the the
  fundamental period (and its inverse) of $p$-adic Hodge theory, up to
  powers of $p$. We establish an exact sequence using $\Aexp$ and
  Frobenii on its filtration, give a link to Fontaine-Laffaille
  modules and the Bloch-Kato fundamental exact sequence and finally
  compute the integral finite part of a lattice of a crystalline
  representation, giving a connection to the local $L$-function of
  $V$.
\end{abstract}

\section{Introduction}

In their fundamental work \cite{bk93}, Bloch and Kato used and developed
many techniques of what is now usually called $p$-adic Hodge theory,
initiated before in large parts by Fontaine. Bloch and Kato's focus was
the development of a general conjecture concerning special values of 
$L$-functions, which culminated in their formulation of a version of
the Tamagawa number conjecture.

Working locally at a fixed prime $p$ and a fixed finite extensions $K
/ \Qp$ with absolute Galois group $G_K$, we take a closer look at the
computations done in sections 3 and 4 in loc.cit., which depend in
certain situations on the property that the $p$-adic representation
$V$ of $G_K$ under consideration is ``in the Fontaine-Laffaille range''. This
is a condition on the filtration of the filtered $\varphi$-module
associated to $V$.

We introduce an auxillary integral ring $\Aexp$, which, after
inverting $p$, computes the module $\Dcris(V)$ of a crystalline
representation, if one fixes a $G_K$-equivariant lattice $T \subset V$.  A
nice property of this integral version is that it contains already (up
to some $p$-powers) the inverses of the fundamental period $t$, so
that no awkward twisting to a positive representation is necessary.
Note that simply inverting $t$ in for example $\Acris$ implies that
$p$ is then also already inverted, which leaves the integral world.

Using this ring, we show that one can construct a finite rank
Fontaine-Laffaille module $\Dexp(T)$ out of $T$, which is used to connect the
$p$-adic valuation of the special value at $s = 0$ of the local
$L$-function $P(V, p^{-s})$ to a certain measure on this Fontaine-Laffaille module
(via Bloch-Kato's fundamental exact sequence), without any condition
on the filtration range of $V$:

\begin{thm}
  Asumme $K / \Qp$ is unramified and let $V$ be a crystalline
  representation. Fix a $G_K$-equivariant lattice $T$ in $V$ and
  assume further that $P(V, 1) \not= 0$. Then:
  \begin{enumerate}
  \item $H^1(\Dexp(T))[1/p] \cong H^1_e(K, V)$.
  \item $H^1(\Dexp(T)) \cong H^1_e(K, T)$.
  \item $\exp_e : \Ddr(V) / \Ddr^0(V) \ra H^1(K, V)$ coincides with the composite map
    \begin{align*}
      (\Dexp(T) / \Dexp^0(T)) [1/p] \overset{1 - \varphi}{\lra}
      (\Dexp(T) / (1 - \varphi^0) \Dexp(T)) [1/p] \\=
      H^1(\Dexp(T))[1/p] \cong H^1(K, V),
    \end{align*}
    where the last canonical identification is explained in the proof.
  \end{enumerate}
\end{thm}
Here, $H^1_e(K, -)$ denotes the exponential part of $H^1(K, -)$, that
is, the image of the Bloch-Kato exponential map. As a corollary, we
obtain:
\begin{cor}
  Let $\mu$ be the Haar-measure on the finite-dimensional $K$-vector
  space $H^1(K, V)$ such that the image of the lattice
  \[
    \Dexp(T) / \Dexp^0(T) \subset \Ddr(V) / \Ddr^0(V) \overset{\sim}{\lra} H^1_e(K, V) = H^1_f(K, V)
  \]
  has measure $1$. Then
  \[
    \mu(H^1_e(K, T)) = | P(V, 1) |_p^{-1}.
  \]
\end{cor}

\section{Basic concepts from $p$-adic Hodge theory}

Fix a prime number $p$. Let $K / \Qp$ be a finite extension of the
$p$-adic numbers, and denote by $G_K$ the absolute Galois group of
$K$. Let $K_0$ be the maximal unramified subextension of $K / \Qp$.
Usually, $V$ will denote a $p$-adic representation, that is, a finite
dimensional $\Qp$-vector space equipped with a continuous and linear
$G_K$-action. Similarly, $T$ will usually denote a $G_K$-stable
$\Zp$-lattice in $V$. Such lattices always exist. One is interested in
the classification of such $V$ and $T$, and Fontaine's rings have
proven to be a powerful tool for this. We refer to \cite{fontouy} as a
basic reference.

Let $\Ocp$ be the ring of integers of the completion of an algebraic
closure $\Cp$ of $\Qp$. Let $\Etplus := \varprojlim_{x \mapsto x^p}
\Ocp / p$, which is a ring of charakteristic $p$, equipped with a
Frobenius $\varphi : x \mapsto x^p$, and a Galois action of $\sigma
\in G_K$ via $(x_n) \mapsto (\sigma(x_n))$. If $x = (x_k) \in
\Etplus$, let $x^{(0)} = \lim_{m \to \infty} \hat{x}_m^{p^m}$, where
$\hat{x}_m \in \Ocp$ is any lift of $x_m \in \Ocp / p$. This 
defines a non-archimedean valuation $v : x \mapsto v_p(x^{(0)})$
on $\Etplus$.

Let $\Atplus = W(\Etplus)$, the ring of Witt vectors of
$\Etplus$. This makes sense since $\Etplus$ is perfect, since it is a 
perfection of the non-perfect ring $\Ocp / p$. $\Atplus$ is 
a ring of charakteristic $0$ with Frobenius $\varphi :
\sum_{n \geq 0} [x_n] p^n \mapsto \sum_{n \geq 0} [\varphi(x_n)] p^n$, 
and an action of $G_K$ that is defined analogously.

One has the important ring homomorphism 
\[
\theta: \Atplus \ra \Ocp,~~~ \sum [x_n] p^n \mapsto \sum x_n^{(0)}
p^n,
\]
which arises conceptually in Fontaine's theory of universal
thickenings.

We fix a system of $p^n$-th roots of unity $\varepsilon^{(n)} \in \Cp$
with $\varepsilon^{(0)} = 1, \varepsilon^{(1)} \not= 1$ and
$(\varepsilon^{(n+1)})^p = \varepsilon^{n}$. Then $\varepsilon =
(\overline{\varepsilon^{(n)}}) \in \Etplus$, where $\overline{x}$
means reduction $\mod p$. Let $\overline{\pi} := \varepsilon - 1$
(this notation is slightly unfortunate, but standard). One can show
that $v(\overline{\pi}) = \frac{p}{p-1}$.
Let $\pi := [\varepsilon] - 1 \in \Atplus$. Observe that $\pi \equiv
\overline{\pi} \mod p$, but $\pi \not= [\overline{\pi}]$.

Further let $\tilde{p} \in \Etplus$ with $\tilde{p}^{(0)} = -p$. Set
$\xi = [\tilde{p}] + p \in \Atplus$, i.e. $\theta(\xi) = -p + p =
0$. One can show that $\ker \theta = (\xi)$, since $\ker \theta
\subset (\xi, p)$, using the fact that $\Atplus$ is $p$-adically
complete and that $\Ocp$ does not have any $p$-torsion. More
generally, if $\xi' \in \Atplus$ such that $\theta(\xi') = 0$ and 
$v(\overline{\xi'}) = 1$, then $\ker \theta = (\xi')$.

One defines $\Bplusdr = \varprojlim_n \Atplus[1/p] / (\ker
\theta)$. $\theta$ extends to $\Bplusdr$, where $\ker \theta = (\xi)$
still holds. Let
\[
t = \log (1+\pi) = \sum_{n \geq 1}(-1)^{n+1}\frac{\pi^n}{n} \in
\Bplusdr,
\]
the fundamental period of $p$-adic Hodge theory. $t$ only depends on
the choice of a compatible system of $p^n$-th roots of unity.
Interestingly, one has $\ker \theta = (t)$. This shows that $\Bplusdr$
is a complete discrete valuation ring with maximal ideal $(t)$. $\Bdr
= \Bplusdr [1/t]$ possesses the filtration $\Fil^i \Bdr = t^i
\Bplusdr, i \in \mb{Z}$. Since $\Bplusdr / (t) \cong \Cp$, induced by
the map $\theta$, one has (algebraically, non-canonically) $\Bdr \cong
\Cp[[t]]$. But observe that the topology on $\Bplusdr$ is defined via
the inverse limit topology and the topology on $\Atplus$, which is
induced via the Witt-construction by the valuation topology on
$\Etplus$. With this topology, one still has a continuous action of
$G_K$, but the action of $\varphi$ does not extend to $\Bplusdr$.

This being the case one considers the ring $\Acris$, which is defined
as the $p$-adic completion of the divided power envelope of $\Atplus$
with respect to the ideal $\ker \theta$, i.e.
\[
\Atplus[\frac{a^m}{m!}; a \in \ker \theta]^{\wedge} = \Atplus \{
\frac{\xi^m}{m!} \} \subset \Bplusdr.
\]
If $x \in \Acris$, then we may write (non-uniquely) $x = \sum_{n} a_n
\frac{\xi^n}{n!}$ with $a_n \in \Atplus$ and $a_n \to 0$ $p$-adically.
The map $\theta$ and the $\varphi$ and $G_K$-action extend to $\Acris$.
Further, $t \in \Acris$, since $\pi = b \xi$ ($\theta(\pi) = 0$) and
\[
  \frac{\pi^n}{n} = (n-1)! b^n \frac{\xi^n}{n!},~~ (n-1)! \to 0.
\]
Let $\Bpluscris = \Acris[1/p]$, $\Bcris = \Bpluscris[1/t] \subset
\Bdr$. $\Bcris^{G_K} = K_0$. Two facts about $\Acris$ are:
$\varphi(\xi) \in p \Acris$, and $t^{p-1} \in p \Acris$, hence $\Bcris
= \Acris[1/t]$. Two caveats about $\Bcris$ are: it has a funny
topology, since one can show that the topology induced on $\Bpluscris$
by $\Bcris$ (which comes equipped with the locally convex final
topology) is not the natural topology on $\Bpluscris$.  Furthermore,
$\Bpluscris \subsetneq \Fil^0 \Bcris$.

Let $V$ now be a $p$-adic representation. Let $\Ddr(V) = (\Bdr
\otimes_{\Qp} V)^{G_K}$. This is a $K$-vector space, and the injectivity
of the canonical map
\[
  \alpha: \Bdr \otimes_K \Ddr(V) \ra \Bdr \otimes_{\Qp} V,~~~ b \otimes (\sum
  b_n \otimes v_n) \mapsto \sum b b_n \otimes v_n
\]
shows that $\dim_K \Ddr(V) \leq \dim_{\Qp} V$. If equality holds, we
call $V$ \textbf{de Rham}. $\Ddr(V)$ comes equipped with a seperated
and exhaustive $K$-vector space filtration, given by $\Fil^i \Ddr(V) =
(\Fil^i \Bdr \otimes V)^{G_K}$.

Similarly, we let $\Dcris(V) := (\Bcris \otimes_{\Qp} V)^{G_K}$. This
is a $K_0$-vector space, and the injectivity of the analogous
$\alpha$-map for $\Bcris$ shows that $\dim_{K_0} \Dcris(V) \leq
\dim_{\Qp} V$. If equality holds, we call $V$
\textbf{crystalline}. $\Dcris(V)$ comes equipped with a $K_0$-linear
$\varphi$-action. Further, $K \otimes_{K_0} \Dcris(V)$ comes equipped
with a $K$-vector space filtration. If $V$ is crystalline then $V$ is
de Rham.

One fundamental theorem of Colmez and Fontaine states: the assignment
$V \mapsto \Dcris(V)$ induces an equivalence of categories between the
crystalline representations of $G_K$ and the category of $K$-filtered
admissible (i.e. $t_H(D) = t_N(D)$ and $t_H(D') \leq t_N(D')$ for all
subobjects $D' \subset D$, where $t_H$ resp. $t_N$ are the Hodge
number resp. the Newton number) $\varphi$-modules.  This equivalence heavily
uses the fact that the map $\alpha$ above in the $\Bcris$-case is
actually an isomorphism.

\section{The period ring $\Aexp$}

\begin{defi}
  Let 
  \[
    \Atplus \left\{ \frac{\pi}{p} \right\} := \Atplus \{X\} / (pX - \pi),
  \]
  where $A\{ X \}$ denotes the $p$-adic completion of $A[X]$ for any
  ring $A$, equipped with quotient topology.
\end{defi}

If $x \in \Atplus \{ \pi / p \}$, we may write (non-uniquely) $x
= \sum_{n \geq 0} a_n (\pi / p)^n$ in $\Atplus[1/p]$. The natural
actions of $\varphi$ and $G_K$ extend to actions on $\Atplus \{
\pi / p \}$.

\begin{lem}
  \label{lempit}
  In $\Atplus \{ \pi / p \} \subset \Bplusdr$ one has the relation 
  $t/p = \pi / p \cdot v$ mit $v \in \Atplus \{ \pi / p \}^\times$, i.e. 
  $\Atplus \{ \pi / p \} = \Atplus \{ t / p \}$. In particular, $t \in \Atplus \{ \pi / p \}$.
\end{lem}
\begin{proof}
  First, we observe that 
  \[
  \frac{t}{p} = \sum_{n \geq 1} (-1)^{n+1} \cdot \frac{\pi^n}{p \cdot n} = \frac{\pi}{p} \cdot
  \left( \sum_{n \geq 1} a_n \left(\frac{\pi}{p} \right)^n \right) = \pi / p \cdot v,
  \]
  with $a_n \to 0$ $p$-adically.

  Now, since $v \mod p \in \Etplus[X] / (\overline{\pi})$ is -1, hence a unit, we have that $v
  \in \Atplus \{ \pi / p \}^\times$. Hence the claim.
\end{proof}

\begin{defi}
  Let $A$ be a subring of $\Bdr$, such that the Frobenius $\varphi$
  acts on $A$ (e.g. $\Acris$). Set
  \[
    \Fil^i_p A := \{ x \in \Fil^i A|~ \varphi(x) \in p^i A \},
  \]
  where $(\Fil^i A)_{i \in \mb{Z}}$ is the filtration induced by $\Bdr$.
\end{defi}

\begin{defi}
  Let $\Aexp := \Atplus\{\pi / p \} [p / t]$. The Frobenius $\varphi$
  on $\Atplus \{ \pi / p \}$ extends to $\Aexp$. We equip $\Aexp$ with 
  the filtration given by 
  \[
  \Fil^i \Aexp := \bigcup_{i + k \geq 0} \left(\frac{p}{t}\right)^{k} \Fil_p^{i +
    k} \Atplus \left\{ \frac{\pi}{p} \right\}.
  \]
  Since $\Fil^0_p \Atplus \{ \pi / p \} = \Atplus \{ \pi / p \}$, this
  filtration is seperated and exhaustive.
\end{defi}

\begin{prop}
  \label{exseq}
For every $k \geq 0$ we have the exact $G_K$-equivariant sequence
\[
0 \lra \left(\frac{t}{p}\right)^k \cdot \Zp \lra \Fil^k_p \Atplus\{ \pi / p \} \overset{1 - p^{-k} \varphi
}{\lra} \Atplus \{ \pi / p \} \lra 0,
\]
which admits a continuous (not necessarely $G_K$-equivariant)
splitting $\Atplus \{ \pi / p \} \ra \Fil^k_p \Atplus \{ \pi / p \}$.
\end{prop}
\begin{proof}
  Obviously, 
  \[
    \left(\frac{t}{p}\right)^k \cdot \Zp \subset \ker (1 - p^{-k} \varphi).
  \]
  On the other hand, if $x \in \ker (1 - p^{-k} \varphi)$ then $x =
  \sum_n a_n (t / p)^n$ (see Lemma \ref{lempit}), with $\Atplus \ni a_n \ra 0$
  $p$-adically. For any $n \in \mb{N}$ we have $(p^{-k} \varphi)^n(x)
  \equiv \varphi^n(a_n) (t/p)^k \mod p \Atplus \left\{ \pi /
    p\right\}$, hence $x = y (t/p)^k$, with $y \in \Atplus$ and
  $\varphi(y) = y$, that is, $y \in \Zp$ as is well-known.

  We now prove that $\Fil^k_p \Aexp$ is the $p$-adic closure of the module
  \[
    \Atplus[\xi^i \cdot (t / p)^j ; i + j \geq k ],
  \]
  which we denote by $N$. If $i + j \geq k$ one has
  \[
    \varphi \left( \xi^i \cdot \left( \frac{t}{p} \right)^j \right) = p^{i+j} \cdot 
    \left(1 + \frac{\pi_0}{p}\right)^i \cdot \left( \frac{t}{p} \right)^j,
  \]
  where $\pi_0$ is the trace from $K((t))$ to $K((t^p))$ of
  $\pi$. Here we recall that $\pi = \exp t - 1 \in K((t))$.  Obviously
  $(t / p)^k \cdot \Zp \subset N$, so we have to prove that for any $y
  \in \Atplus \left\{ \pi / p \right\}$ there exists an $x \in
  \Fil^k_p \Atplus \left\{ \pi / p \right\}$ with $(1 - p^{-k}
  \varphi)(x) = y$. Since $N$ and $\Atplus \left\{ \pi / p \right\}$
  are seperated and complete with respect to the $p$-adic topology, it
  suffices to show this result $\mod p$. 
  If $y = \sum_{n > k} a_n (t / p)^n$  then $x = -y$ will do the job.

  Thus it remains to show that if $y \in \Atplus$ and $j \leq $k there
  exists an $x \in N$ such that 
  \[
    (1 - p^{-k} \varphi)(x) - y \left( \frac{t}{p} \right)^j \in 
    \left( p, \left( \frac{t}{p}\right)^i; i > n  \right) \unlhd 
    \Atplus \left\{ \pi / p \right\}.
  \]
  One checks that 
  \[
    x = y' (q')^{k - j} \left( \frac{t}{p} \right)^{i},
  \]
  where $y' \in \Atplus$ is a solution of
  \[
    \varphi(y') - q'^{k - j} y' = b
  \]
  in $\Atplus$, satisfies this property (recall that $q' = \varphi^{-1}(q)$).
\end{proof}

\begin{cor}
  Dividing out $(t/p)^{-k}$ and taking the the direct limit over the
  sequence in Proposition \ref{exseq} we obtain an exact sequence
  \[
    0 \lra \Zp \lra \Fil^0 \Aexp \overset{1 - \varphi}{\lra} \Aexp \lra 0,
  \]
  where $\varphi$ is the extension of the $\varphi$ on $\Atplus$.
\end{cor}

\begin{prop}
  \label{invaexp}
  $(\Aexp)^{G_K} = \mc{O}_{K_0}$
\end{prop}
\begin{proof}
  Obviously, $\mc{O}_{K_0} \subset (\Aexp)^{G_K}$, since
  $(\Atplus)^{G_K} = \mc{O}_{K_0}$. We have the inclusions
  \[
  \varphi\left( \Atplus \left\{ \frac{\pi}{p} \right\} \right) =
  \Atplus \left\{ \frac{\pi^p}{p} \right\} \subset \Acris \subset
  \Atplus \left\{ \frac{\pi}{p} \right\},
  \]
  which leads to 
  \[
    \varphi( \Aexp[1/p] ) \subset \Bcris \subset \Aexp[1/p].
  \]
  Since $(\varphi(\Bcris))^{G_K} = (\Bcris)^{G_K} = K_0$, we have that
  $(\varphi( \Aexp[1/p] ))^{G_K} = K_0$. Since $\varphi$ is injective
  on $\Aexp$, and hence on $\Aexp[1/p]$, we obtain $(\Aexp[1/p])^{G_K} = K_0$. 
  
  Note that by the above exact sequence, $1/p \not \in \Aexp$:
  otherwise $1/p^n \in \Aexp$ for all $n \geq 0$. But taking
  $G_{\Qp}$-invariants gives an injection
  \[
  \Aexp \hookrightarrow H^1(G_{\Qp}, \Zp) = \Hom_{\trm{cts}}(G_{\Qp},
  \Zp).
  \]
  The $\Zp$-module on the right hand side is finitely generated, which
  would lead to a contradiction.

  Alternatively, one can use the exact sequence of \ref{exseq}, the filtration on $\Aexp$ and a
  limit argument to proceed as in the proof of the statement
  \[
    H^0(K, \Fil^i \Bdr / \Fil^j \Bdr) = K,
  \]
  if $i \leq 0 < j$.
\end{proof}

\begin{defi}
  Let $T$ be a full $\Zp$-lattice of $V$ that is invariant under the
  action of $G_K$ (such lattices always exist). We define the modules
  \[
  \Dexp(T) := (\Aexp \otimes_{\Zp} T)^{G_K}
  \]
  and 
  \[
    \Dexp^0(T) := (\Fil^0 \Aexp \otimes_{\Zp} T)^{G_K}.
  \]
\end{defi}

\begin{prop}
  \label{dexpfd}
  If $T$ is as before, $\Dexp(T)$ is free $\mc{O}_{K_0}$-module of
  finite rank less or equal to $\trm{rk}_{\Zp} T$.
\end{prop}
\begin{proof}
  This is a variaton of the proof one usually encounters in Fontaine's
  theory of $B$-admissible rings. We outline the idea:

  Let $B$ be a topological integral domain, equipped with a continuous action of a topological
  group $G$.  Set $C = \trm{Frac}(B)$ and $S = B^G$, which is again an integral domain, and fix
  a closed subring $R \subset S$. Assume $T$ is a finite free $R$-module with continuous
  $G$-action, so that $V = \trm{Frac}(R) \otimes_R T$ is a finite-dimensional
  $\trm{Frac}(R)$-vector space. Set $D_B(T) = (B \otimes_R T)^{G}$, $D_C(V) = (C
  \otimes_{\trm{Frac(R)}} V)^G$ and assume $C^G = \trm{Frac}(S)$. 
  We want to prove the injectivity of the map 
  \[
  B \otimes_{S} D_B(T) \lra B \otimes_R T .
  \]
  The inclusion $B \hookrightarrow C$ and the freeness of $T$ gives a diagram
  \[
    \xymatrix{
    B \otimes_{S} D_B(T) \ar[r] \ar@{^{(}->}[d] & B \otimes_R T \ar@{^{(}->}[dd] \\
    B \otimes_{S} D_C(V) \ar@{^{(}->}[d] \\
    C \otimes_{\trm{Frac}(S)} D_C(V) \ar[r] & C \otimes_{\textrm{Frac}(R)} V
    },
  \]
  so that we are reduced to the case where all the rings are fields. Now one proceeds exactly
  as in \cite{fontouy}, Theorem 2.13.

  The above situation applies with $B = \Aexp$, $R = \Zp \subset \mc{O}_{K_0} = S$. The
  injectivity of the above map implies, by using the above notation and going to the quotient
  field $C$, that $D_B(T)$ is of $S$-rank smaller or equal than the $R$-rank of $T$. Since we
  these latter rings are discrete valuation rings, we are done.
\end{proof}

\begin{prop}
  \label{dexpdcris}
  If $V$ and $T$ are as above, we have
  \[
    \Dexp(T)[1/p] = \Dexp(V) = (\Aexp[1/p] \otimes_{\Qp} V)^{G_K} = \Dcris(V), 
  \]
  and this identification is compatible with the $K$-filtrations and
  the action of the Frobenius $\varphi$.
\end{prop}
\begin{proof}
  The proof is similarly as in Proposition \ref{invaexp}. We have inclusions
  \[
  (\varphi(\Aexp) \otimes T)^{G_K} \subset (\Acris
  \otimes T)^{G_K} \subset (\Aexp \otimes T)^{G_K},
  \]
  and since $\varphi$ is bijective on $\Dcris(V)$ and injective on $\Aexp$, we conclude as
  before $(\Aexp[1/t] \otimes T)^{G_K} = \Dcris(V)$.

  The compatibility with filtration and Frobenius can be checked by the construction.
\end{proof}

\section{Categories in integral $p$-adic Hodge theory}

Let $K / \Qp$ be unramified for this section.

\begin{defi}
  A \textbf{Fontaine-Laffaille module} over $\OK$ is a triple $(M, ( M^i )_{i \in \mb{Z}},\\
  (\varphi^i)_{i \in \mb{Z}} )$, which we also denote simply by $M$, consisting of
  \begin{itemize}
    \item an $\OK$-module $M$,
    \item an exhausting and seperated decreasing filtration (of $\OK$-modules) $M^i$ of $M$,
    \item a family of $\sigma$-semilinear maps $\varphi^i : M^i \ra M$
      with the property $\varphi^i|_{M^{i+1}} = p \cdot \varphi^{i+1}$.
  \end{itemize}
  A morphism of Fontaine-Laffaille modules $f: M \ra N$ is a
  $\OK$-linear map $f$ such that $f(M^i) \subset N^i$ and $f \circ
  \varphi^i_M = \varphi^i_N \circ f$.  We denote by
  $\underline{\trm{MF}}_{\OK}$ the exact category of all
  Fontaine-Laffaille modules over $\OK$.
\end{defi}

\begin{defi}
  A \textbf{filtered Dieudonn\'e module} over $\OK$ is a Fontaine-Laffaille module $M$ such that 
  \begin{itemize}
    \item $M$ is of finite type over $\OK$, 
    \item $M^i = M$ for $i \ll 0$ and $M^j = 0$ for $j \gg 0$,
    \item $M_\sigma = \sum_{i \in \mb{Z}} \varphi^i(M^i)$.
  \end{itemize}
  We denote by $\underline{\trm{MF}}^{\trm{fd}}_{\OK}$ (fortement divisible) the category of
  all filtered Dieudonn\'e modules.
\end{defi}

Here, $M_\sigma$ denotes the underlying module $M$, where $W$ acts via
$\sigma$.  Note that contrary to the usual convention we allow our
Dieudonn\'e modules to contain torsion.

\begin{ex}
  We have $\Aexp = (\Aexp, (\Fil^i \Aexp)_{i \in \mb{Z}},
  (\varphi^i)_{i \in \mb{Z}}) \in \MF$, where $\Aexp$ and the
  filtration are given as before, and $\varphi^i = 1/p^i \cdot
  \varphi$, with $\varphi$ induced from the Frobenius on $\Atplus$.
\end{ex}

\begin{thm}
  The category $\underline{\trm{MF}}^{\trm{fd}}_{\OK}$ is abelian.
\end{thm}
\begin{proof}
  This follows from the fact that $\underline{\trm{MF}}^{\trm{fd},
    \trm{tor}}_{\OK}$, the subcategory of all torsion $\OK$-modules
  (\cite{fontlaf}, Proposition 1.8), is abelian, and completeness: let $f: M
  \ra N$ be a map in $\MF_{\mc{O_K}}$ with $M, N \in
  \underline{\trm{MF}}^{\trm{fd}}_{\OK}$. This gives us, for any $n
  \in \mb{N}$, a map $f_n : M / p^n \ra N / p^n$ in
  $\underline{\trm{MF}}^{\trm{fd}, \trm{tor}}_{\OK}$, since $M$ and
  $N$ are of finite type, hence kernel and cokernel of $f_n$ exist.

  Since $M = \varprojlim M / p^n$, in a compatible way with the
  filtration and the Frobenii, we obtain, by going to the limit, the
  kernel and the cokernel of the map $f$. The normality of mono- and
  epimorphisms is an easy consequence again of the property that
  $\underline{\trm{MF}}^{\trm{fd}, \trm{tor}}_{\OK}$ is abelian.
\end{proof}


\begin{defi}
  A \textbf{filtered $\varphi$-module} over $K$ is a triple $(D, (D^i)_{i \in \mb{Z}}, \varphi)$, 
  consisting of 
  \begin{itemize}
  \item a $K$-vector space $D$,
  \item an exhausting and seperated decreasing filtration (of $K$-vector spaces) $D^i$ of $D$,
  \item a $\sigma$-semilinear map $\varphi: D \ra D$.
  \end{itemize}
  A morphism of filtered $\varphi$-modules is, similarly as before, a morphism of $K$-vector spaces 
  compatible with the filtration and Frobenius $\varphi$. We denote by $\MF_{K}$ the category of 
  all $\varphi$-modules.

  A finite-dimensional filtered $\varphi$-module is called
  \textbf{admissible} if $t_N(D) = t_H(D)$ and $t_N(D') \leq t_H(D')$
  for all subobjects $D' \subset D$ in $\MF_{K}$, where $t_N$
  resp. $t_H$ are the Newton- resp. Hodge number of $D$ (see \cite{fontouy}, 6.4.2.).
\end{defi}

\begin{ex}
  If $M \in \MF_{\OK}^{\trm{fd}}$, one can naturally associate a
  finite-dimensional $\varphi$-module $D$ to $M$, namely $D := K
  \otimes_{\OK} M$, with the filtration induced by $M^i$, and
  Frobenius $\varphi := 1/p^n \cdot \varphi^n$ for $n \ll 0$. We call
  $M$ \textbf{admissible} if $D$ is admissible.
\end{ex}

\begin{prop}
  \label{critfd}
  Let $D \in \MF_{K}$ be admissible. Then an $\OK$-lattice $M$,
  equipped with a filtration $M^i$ such that $M^i[1/p] = D^i$ can
  be considered as an object of $\MF^{\trm{fd}}_{\OK}$ if and only if
  $\varphi(M^i) \subset p^i M$ (that is, one puts $\varphi^i := p^{-i}
  \varphi$).
\end{prop}
\begin{proof}
  The only thing we have to check is that the condition on $\varphi$
  holds, then $D$ is already in $\MF^{\trm{fd}}_{\OK}$.  This can be
  infered by the proof of \cite{fontlaf}, Theorem 3.2, after reducing
  to the case where all weights are $\geq 0$ (see also \cite{laf},
  Proposition 7.8).
\end{proof}

\begin{prop}
  If $V$ is crystalline and $T$ as above, then $\Dexp(T) \in \MF_{\OK}^{\trm{fd}}$.
\end{prop}
\begin{proof}
  We know from proposition \ref{dexpfd} that $\Dexp(T)$ is a free $\OK$-module of finite
  rank and that $\Dexp(T)[1/p] = \Dcris(V)$ (\ref{dexpdcris}) is admissible,
  since $V$ is crystalline. Since $\Dexp^i(T)[1/p] = \Fil^i \Dcris(V)$
  and $\varphi^i (\Dexp^i(T)) \subset \Dexp(T)$, the requirements of
  proposition \ref{critfd} are fulfilled, hence the claim.
\end{proof}

\section{Computation of $H^1_e(K, T)$}

Assume again that $K / \Qp$ is unramified. We collect some facts
from sections 3 and 4 of \cite{bk93}.

\begin{prop}
  Let $M \in \MF_{\OK}^{\trm{fd}}$ and put
  \[
    H^0(M) = \trm{ker} ((1 - \varphi^0) : M^0 \ra M),~~~
    H^1(M) = \trm{coker} ((1 - \varphi^0) : M^0 \ra M).
  \]
  and $H^i(M) = 0$ for $i \geq 2$. Then $(H^i)_{i \in \mb{N}}$ is a 
  cohomological $\delta$-functor.
\end{prop}
\begin{proof}
  This is abundantly clear by the snake lemma.
\end{proof}

Recall (\cite{bk93}, Proposition 1.17) the Bloch-Kato fundamental exact
sequences
\[
0 \lra \Qp \lra \Bcris^{\varphi = 1} \oplus \Bplusdr \overset{f}{\lra}
\Bdr \lra 0,
\]
where $f(x, y) = x - y$, and 
\[
0 \lra \Qp \lra \Bcris \oplus \Bplusdr \overset{f}{\lra} \Bcris \oplus
\Bdr \lra 0,
\]
where $f(x,y) = ((1 - \varphi)(x), x - y)$. If $V$ is a $p$-adic 
representation, we call the maps
\[
  \exp_e: \Ddr(V) \overset{\delta}{\ra} H^1(K, V),~~~
  \exp_f: \Dcris(V) \oplus \Ddr(V) \overset{\delta}{\ra} 
  H^1(K, V)
\]
the \textbf{Bloch-Kato exponential maps}, which are induced by the
connecting homomorphism of Galois cohomology. We set
\[
  H^1_e(K, V) := \trm{Im}(\exp_e),~~~ H^1_f(K, V) := \trm{Im}(\exp_f). 
\]
If $T \subset V$ is a $G_K$-equivariant $\Zp$-lattice in $V$, denote
by $\iota: H^1(K, T) \ra H^1(K, V)$ the canonical map. Set
\[
  H^1_f(K, T) := \trm{Im}(\exp_f)
\]

Recall that the local $L$-function in the case $p = l$ is defined as 
\[
  P(V, X) := \trm{det}_{K}(1 - f \cdot X|~ \Dcris(V)),
\]
where $f$ denotes the $K$-linear map $\varphi^{[K : \Qp]}$.

Assume now that $P(V, 1) \not= 0$. Then (cf.~the first lines in the
proof of Theorem 4.1, \cite{bk93}) $H^1_f(K, V) = H^1_e(K, V)$.

\begin{thm}
  Let $V$ be a crystalline representation, fix a $G_K$-equivariant
  lattice $T$ in $V$, and assume $P(V, 1) \not= 0$. Then:
  \begin{enumerate}
  \item $H^1(\Dexp(T))[1/p] \cong H^1_e(K, V)$.
  \item $H^1(\Dexp(T)) \cong H^1_e(K, T)$.
  \item $\exp_e : \Ddr(V) / \Ddr^0(V) \ra H^1(K, V)$ coincides with the composite map
    \begin{align*}
      (\Dexp(T) / \Dexp^0(T)) [1/p] \overset{1 - \varphi}{\lra} (\Dexp(T) / (1 - \varphi^0) \Dexp(T)) [1/p]\\
      =  H^1(\Dexp(T))[1/p] \cong H^1(K, V),
    \end{align*}
    where the last canonical identification is explained in the proof.
  \end{enumerate}
\end{thm}
\begin{proof}
  The proof is similar to \cite{bk93}, Lemma 4.5.
  We have the following commutative diagram:
  \[
\xymatrix{
  0 \ar[r] & \Zp \ar[d] \ar[r] & \Fil^0 \Aexp \ar[r] \ar[d]^{x \mapsto (x,x)} & \Aexp \ar[r] 
  \ar[d]^{x \mapsto (x,0)} & 0 \\
  0 \ar[r] & \Qp \ar[r] & \Aexp[1/p] \oplus \Bplusdr \ar[r] &
  \Aexp[1/p] \oplus \Bdr \ar[r] & 0 }.
\]
Inverting $p$, tensoring with $V$ over $\Qp$ and taking $G_K$-cohomology, we obtain a diagram
\[
\!\!\!\!\!\!\xymatrix@C=8pt{ 0 \ar[r] & H^0(\Dexp(T))[1/p] \ar[r] \ar[d]^{=} & \Dexp^0(T)[1/p] \ar[r] \ar[d] &
  \Dexp(T)[1/p] \ar[d] \ar[r] & H^1(\Dexp(T))[1/p] \ar[r] \ar[d] & 0 \\
  0 \ar[r] & H^0(K, V) \ar[r] & \Dcris(V) \oplus \Ddr^+(V) \ar[r] & \Dcris(V) \oplus \Ddr(V)
  \ar[r] & H^1(K, V) }
\]
with exact rows. Since $\Dexp(T)[1/p] = \Dcris(V) = \Ddr(V)$ and
$\trm{Im}(\exp_f) = \trm{Im}(\exp_e)$, we have the claimed
identification $H^1(\Dexp(T))[1/p] \cong H^1_f(K, V)$.

The exact sequence
\[
  0 \lra \Dexp(T) \overset{p}{\lra} \Dexp(T) \lra \Dexp(T) / p \lra 0
\]
(which exists in $\MF^{\trm{fd}}_{\OK}$, since $\Dexp(T) \in
\MF^{\trm{fd}}_{\OK}$) induces a sequence
\[
  H^0(\Dexp(T) / p) \lra H^1(A \otimes T) \lra H^1(A \otimes T) \lra
  H^1(A \otimes T / p).
\]
Recall that $H^1_f(K, T) = \iota^{-1} (H^1_f(K,V))$, where $\iota :
H^1(K, T) \ra H^1(K, V)$.  Since $H^1(\Dexp(T)) \subset H^1(K, T)$ it
suffices to show that the cokernel of this inclusion does not have any
$p$-torsion. But this follows from the the commutative diagram
\[
\xymatrix{
  H^0(\Dexp(T) / p) \ar[r] \ar[d]^\cong & H^1(A \otimes T) \ar[r]^-{p} \ar[d]^{\subset} &  
  H^1(A \otimes T) \ar[r] \ar[d]^{\subset} & H^1(A \otimes T / p) \ar[d]^{\subset} \\
  H^0(T / p) \ar[r] & H^1(T) \ar[r]^p & H^1(T) \ar[r] & H^1(T / p)  }.
\]
\end{proof}

\begin{cor}
  Let $\mu$ be the Haar-measure on the finite-dimensional $K$-vector
  space $H^1(K, V)$ such that the image of the lattice
  \[
    \Dexp(T) / \Dexp^0(T) \subset \Ddr(V) / \Ddr^0(V) \overset{\sim}{\lra} H^1_e(K, V) = H^1_f(K, V)
  \]
  has measure $1$. Then
  \[
    \mu(H^1_f(K, T)) = | P(V, 1) |_p^{-1}.
  \]
\end{cor}
\begin{proof}
  This follows from the definition of $P(V, X)$ and the b), c) from the previous theorem.
\end{proof}

\bibliography{local} \bibliographystyle{plain}

Andreas Riedel

Universit\"at Heidelberg, Mathematisches Institut

Im Neuenheimer Feld, 205

69115 Heidelberg, Germany

\vspace{2mm}

Email: \texttt{ariedel@mathi.uni-heidelberg.de}

Phone: 0049 6221 14231

\end{document}